\documentclass[11pt,reqno]{amsart}
\usepackage{color}
\usepackage{amsmath,amssymb,amsfonts,amsthm,bm}
\usepackage{mathrsfs}
\usepackage{graphicx}
\usepackage{enumerate}
\usepackage[numbers, sort&compress]{natbib}
\usepackage[shortlabels]{enumitem}
 \usepackage{a4wide}
\newtheorem{theorem}{Theorem}[section]

\newtheorem{lemma}[theorem]{Lemma}

\theoremstyle{assumption}

\theoremstyle{definition}

\theoremstyle{remark}
\newtheorem{remark}[theorem]{Remark}
\numberwithin{equation}{section}

\newcommand{\eps}{\varepsilon} 
\newcommand{\norm}[1]{\Vert#1\Vert}
\newcommand{\abs}[1]{\left\vert#1\right\vert}

\newcommand{\inner}[1]{\left(#1\right)}
\newcommand{\comi}[1]{\left<#1\right>}

\newcommand{\normm}[1]{{ \vert\kern-0.25ex \vert\kern-0.25ex \vert #1 
		\vert\kern-0.25ex \vert\kern-0.25ex \vert}}

\makeatletter

\@namedef{subjclassname@2020}{%
	\textup{2020} Mathematics Subject Classification}

\def\@startsection#1#2#3#4#5#6{%
	\if@noskipsec \leavevmode \fi
	\par \@tempskipa #4\relax
	\@afterindentfalse
	\ifdim \@tempskipa <\z@ \@tempskipa -\@tempskipa \@afterindentfalse\fi
	\if@nobreak \everypar{}\else
	\addpenalty\@secpenalty\addvspace\@tempskipa\fi
	\@ifstar{\@dblarg{\@sect{#1}{\@m}{#3}{#4}{#5}{#6}}}%
	{\@dblarg{\@sect{#1}{#2}{#3}{#4}{#5}{#6}}}%
}

\def\@settitle{%
	\bgroup
	\centering
	\vglue1cm
	\fontseries{b}\selectfont
	\uppercasenonmath\@title
	\@title
	\vskip20pt plus 6pt minus 8pt
	\egroup
}

\def\@setauthors{%
	\begingroup
	\trivlist
	\centering \bfseries
	\normalsize\@topsep30\p@\relax
	\advance\@topsep by -\baselineskip
	\item\relax
	\andify\authors
	{\rmfamily\authors}%
	\endtrivlist
	\endgroup
}

\def\@setaddresses{\par
	\nobreak \begingroup
	\normalsize
	\def\author##1{\nobreak\addvspace\bigskipamount}%
	\def\\{\unskip, \ignorespaces}%
	\interlinepenalty\@M
	\def\address##1##2{\begingroup
		\par\addvspace\bigskipamount\noindent
		\@ifnotempty{##1}{(\ignorespaces##1\unskip) }%
		{\ignorespaces##2}\par\endgroup}%
	\def\curraddr##1##2{\begingroup
		\@ifnotempty{##2}{\nobreak\indent{\itshape Current address}%
			\@ifnotempty{##1}{, \ignorespaces##1\unskip}\/:\space
			##2\par}\endgroup}%
	\def\email##1##2{\begingroup
		\@ifnotempty{##2}{\nobreak\noindent{\itshape E-mail address}%
			\@ifnotempty{##1}{, \ignorespaces##1\unskip}\/: 
			##2\par}\endgroup}%
	\def\urladdr##1##2{\begingroup
		\@ifnotempty{##2}{\nobreak\indent{\itshape URL}%
			\@ifnotempty{##1}{, \ignorespaces##1\unskip}\/:\space
			\ttfamily##2\par}\endgroup}%
	\addresses
	\endgroup
}

\renewcommand\section{\@startsection{section}{1}{\z@}%
	{27pt plus 6pt minus 8pt}{14pt plus 6pt minus 8pt}%%
	{\center\normalfont\large\bfseries}}

\makeatother

\begin{document}
	
\title[MHD Boundary layer without viscosity]{Well-posedness in Sobolev spaces  of the two-dimensional MHD Boundary layer equations without   viscosity}

\author[W.-X. Li and R. Xu]{Wei-Xi Li  \and Rui Xu}
	
	\date{}

	\address[W.-X. Li]{ 
		School of Mathematics and Statistics,  
 and Computational Science Hubei Key Laboratory,   Wuhan University,  430072 Wuhan, China
	}

	\email{
		wei-xi.li@whu.edu.cn}
		
\address[R. Xu]{ 
		School of Mathematics and Statistics,     Wuhan University,  430072 Wuhan, China
	}		
\email{xurui218@whu.edu.cn}



\maketitle

\begin{abstract}
We consider the two-dimensional  MHD Boundary layer system without hydrodynamic viscosity,  and establish  the existence and uniqueness of solutions in Sobolev spaces  under the assumption that the tangential component of magnetic fields dominates.  This gives a complement to the previous works of Liu-Xie-Yang [Comm. Pure Appl. Math. 72 (2019)] and Liu-Wang-Xie-Yang [J. Funct. Anal. 279 (2020)], where the well-posedness theory was established for the MHD boundary layer systems with both viscosity and resistivity and with viscosity only, respectively.  
We use the pseudo-differential calculation,  to overcome a new difficulty arising from the treatment of boundary integrals due to the absence of the diffusion property for the velocity. 
 
\end{abstract}

\section{Introduction }

In this work we study the existence and uniqueness of solution to the  two-dimensional magnetohydrodynamic (MHD) boundary layer system  without viscosity which reads,  letting $\Omega: =\mathbb T\times\mathbb R_+=\big\{ (x,y)\in\mathbb R^2;\  x\in\mathbb T,  y>0\big \}$ be the fluid domain,  
\begin{equation}\label{mhdsys+++}
\left\{
\begin{aligned}
&\partial_t   u +(u\partial_x  +v\partial_{y}) u -(f\partial_x  +g\partial_{y} )f +\partial_x p=0,\\
&\partial_t   f +(u\partial_x f+v\partial_{y}) f-(f\partial_x  +g\partial_{y} )u-\mu\partial_{y}^{2} f=0,\\
& \partial_t g +(u\partial_x +v\partial_y )g  -\partial_y^2 g= f\partial_xv -g\partial_xu,\\
&\partial_x u+\partial_{y} v=0,\quad \partial_x f+\partial_{y} g=0, \\
&\big( v, \partial_{ y}f, g\big)|_{ y=0}= (0, 0,0),  \quad \lim_{y\rightarrow +\infty}(u,f)=(U,B),\\
&u|_{ t=0}=u_{0},\quad f|_{ t=0}=f_{0}\\
\end{aligned}
\right.
\end{equation}
where $ (u, v) $ and $(f, g)$ stand for  
velocity and magnetic fields, respectively,  and $\mu$  is resistivity coefficients,    and  $U, B$ and $p$ are the values on the boundary of the tangential velocity, magnetic fields and pressure, respectively, in the ideal MHD system satisfying the Bernoulli's law:
 \begin{equation}\label{berlaw}
\left\{
\begin{aligned}
&\partial_t U + U \partial_x U+\partial_x p=B \partial_x B,\\
&\partial_t B + U \partial_x B=B \partial_x U.\\
\end{aligned}
\right.
\end{equation}
Note the MHD boundary system with a nonzero hydrodynamic viscosity will reduce to the classical Prandtl equations in the absence of a magnetic field, and the main difficulty for investigating Prandtl equation lies in  the nonlocal property coupled with the loss of one order tangential derivative when dealing with the terms $v \partial_y u$.   The mathematical study on   the Prandtl boundary layer has a  long history, and  there have been extensive works concerning its well/ill-posedness theories.   So far  the two-dimensional  (2D) Prandtl equation  is  well-explored in various function spaces,  see, e.g., \cite{MR3327535, MR3795028,  MR3925144,  MR1476316, MR2601044, MR3429469, MR2849481,  MR3461362, MR3284569,     MR3493958, MR4055987,  MR2020656, MR3710703,MR3464051} and the references therein.    
Compared with the Prandtl equation the treatment is more complicated since we have a new difficulty caused by the additional loss of tangential derivative in the magnetic field.   So far the MHD boundary layer system  is mainly  explored in the two settings. 

\begin{itemize}
\item   Without any structural assumption on initial data the well-posedness for 2D and 3D MHD boundary systems  was   established in Gevrey space by   the first author and T. Yang  \cite{2020arXiv200906513L}  with Gevrey index   up to $3/2$, and it remains interesting to relax the Gevrey index therein to $2$ inspired the previous works of  \cite{MR3925144,lmy}  on the well-posedness for the  Prandtl equations  in Gevrey space with optimal index $2$.  

	\item Under the structural assumption that  the tangential  magnetic field dominates, i.e.,    $f\neq 0$,  the well-posedness in weighted Sobolev space was established by   Liu-Xie-Yang  \cite{MR3882222} and Liu-Wang-Xie-Yang \cite{MR4102162}  without Oleinik's monotonicity assumption,  where the two cases  that with both viscosity and resistivity  and  with only viscosity    are considered, respectively.  These works,  together with the essential role of the Oleinik's monotonicity   for  well-posedness theory of the Prandtl equations (see, e.g., \cite{MR1697762,MR3327535,MR3385340}),   justify  the stabilizing effect of the magnetic field on MHD boundary layer,  no matter whether or not there is resistivity in the magnetic boundary layer equation
\end{itemize}

The aforementioned works \cite{MR3882222, MR4102162} investigated the well-posedness  for MHD  boundary layer system  with  the nonzero   viscosity coefficient. This work aims to consider the case without   viscosity coefficient,  giving  a complement to the previous works \cite{MR3882222, MR4102162}.     To simply the argument we will assume without loss of generality   that $ \mu=1 $ and $( U,  B)\equiv  0$ in  the system \eqref{mhdsys+++} since the result will hold true in the general case if we use some kind of the nontrivial weighted
functions similar to those used in  for the Prandtl equation.   Hence  we consider the following    2D MHD boundary layer system in the region $\Omega=\mathbb T\times\mathbb R_+$
\begin{equation}
\label{mhd+}
\left\{
\begin{aligned}
&\big(\partial_t  +u\partial_x +v\partial_y \big) u-(f\partial_x+g\partial_y)f=0,\\
&\big(\partial_t  +u\partial_x +v\partial_y   -\partial_y^2\big) f-(f\partial_x+g\partial_y)u=0,\\
&\big(\partial_t  +u\partial_x +v\partial_y   -\partial_y^2\big) g= f\partial_xv -g\partial_xu,\\
&\partial_xu+\partial_yv=\partial_xf+\partial_yg=0,\\
& (v,\partial_yf,g)|_{y=0}=(0,0,0), \quad (u, f)|_{y\rightarrow +\infty}=(0,0),\\
& (u,f)|_{t=0}=(u_0, f_0).
\end{aligned}
\right.
\end{equation} 
 By the boundary condition and divergence-free condition above,  we have
\begin{equation*}
v(t,x,y)=- \int_0^y  \partial_x u(t,x,\tilde y) d\tilde y,\quad g(t,x,y)=-\int_0^y  \partial_x f(t,x,\tilde y) d\tilde y.
\end{equation*}
We remark that the  equation for $g$   in \eqref{mhd+}  can be derived from the one for $f$
and    the main difficulty in analysis is the loss of $x$-derivatives in the two terms $v$ and $g.$   As to be seen in the next Section \ref{secderivation}, 
The system \eqref{mhd+} can be  derived from the MHD system 
\begin{equation}\label{mhdv}
\left\{
\begin{aligned}
&\partial_t \bm  u^{\varepsilon} +(\bm  u^{\varepsilon} \cdot\nabla) \bm u^{\varepsilon}-(\bm H^{\varepsilon} \cdot\nabla) \bm H^{\varepsilon} +\nabla   P^{\varepsilon} =0,\\
&\partial_t \bm  H^{\varepsilon}  +(\bm u^{\varepsilon}\cdot\nabla)\bm H^{\varepsilon} =(\bm H^{\varepsilon}\cdot\nabla)\bm u^{\varepsilon}+\mu\varepsilon \Delta \bm H^{\varepsilon},\\
&\nabla\cdot \bm u^{\varepsilon}=\nabla\cdot \bm H^{\varepsilon}=0,\\
\end{aligned}
\right.
\end{equation}
where   $ \bm u^{\varepsilon}=( u^{\varepsilon},   v^{\varepsilon}) , \bm H^{\varepsilon}=(f^{\varepsilon},   g^{\varepsilon})$ denote velocity and magnetic field, respectively.  The MHD system \eqref{mhdv} is complemented with   the   boundary condition that 
\begin{equation*}\label{cond2}
 v^{\varepsilon}|_{y=0}= 0,  \quad \bm (\partial_{y} f^{\varepsilon},   \ g^{\varepsilon}) |_{y=0}=(0, 0).
\end{equation*}
It is an important  issue in both mathematics and physics to ask  the high Reynolds number limit for MHD systems, and  so far it is   justified mathematically  by Liu-Xie-Yang \cite{MR3975147} with the presence of viscosity and the other cases remain unclear.       

{\bf Notation}. Before stating the main result we first list some notation used frequently in this paper.   Given  the domain $\Omega=\mathbb T\times\mathbb R_+$,   we will use  $\norm{\cdot}_{L^2}$ and $\inner{\cdot, \cdot}_{L^2}$ to denote the norm and inner product of  $L^2=L^2(\Omega)$   and use the notation   $\norm{\cdot}_{L_x^2}$ and $\inner{\cdot, \cdot}_{L_x^2}$  when the variable $x$ is specified. Similar notation  will be used for $L^\infty$.  Moreover, we use $L^p_x(L^q_y) = L^p (\mathbb T; L^q(\mathbb R_+))$ for the classical Sobolev space.
Let   $H^m=H^m(\Omega)$  be the standard Sobolev space and define the weighted Sobolev space $H_\ell^m$  by setting,  for  $ \ell\in \mathbb{R} $,
\begin{equation*} 
H_\ell^m=\Big\{f(x,y):\Omega\rightarrow\mathbb{R}; \quad \norm{f}_{H^{m}_{\ell}}^{2}:=
\sum _ { i+j \leq m } \norm{\comi{y} ^ { \ell +j } \partial _ { x } ^ {i } \partial _ { y } ^ { j } f ( x , y )}  ^{2}_{L^{2}}< +\infty \Big\},
\end{equation*}
where here and below $\comi{y}=\big(1+\abs y^2\big)^{1/2}.$
With the above notation the well-posedness theory of  \eqref{mhdsys+++} in weighted Sobolev space can be stated as below. Here the main assumption is that  the tangential magnetic field in \eqref{mhdsys+++}  dominates, that is, $f\neq 0$.

\begin{theorem}\label{pro1+}
	Let $ \ell >\frac 12 $ and $\delta>\ell+\frac 12$ be two given  numbers. Suppose the initial data $  u_{0}, f_{0}$ of  \eqref{mhd+}  lie in $H^{4}_{\ell}(\Omega)$ satisfying   that there exists a constant $c_0>0$  such that 
	for any $(x,y)\in\Omega$,
	 \begin{equation} \label{inco}
   	 	    f_0(x,y) \geq  c_0  \comi y^{-\delta}   \ {\rm  and } \   \sum_{j\leq 2} \abs{\partial_y^j f_0(x,y)}\leq  c_0^{-1} \comi y^{-\delta-j}.
      \end{equation}
 Then  the MHD boundary layer system \eqref{mhd+} admits  a     unique local-in-time solution 
	\begin{equation*} 
	 u,f\in  L^{\infty}([0,T];{H}^{4}_{\ell})
	\end{equation*}	
	for some $T>0$.  Moreover a constant $c>0$ exists such that for any $(t,x,y)\in [0,T]\times \Omega$,    
	\begin{eqnarray*}
	   f(t, x,y) \geq  c  \comi y^{-\delta}   \ {\rm  and } \   \sum_{j\leq 2} \abs{\partial_y^j f(t, x,y)}\leq  c^{-1} \comi y^{-\delta-j}.	
	\end{eqnarray*}
	 \end{theorem} 
 
 \begin{remark}
 	The result above confirms  that the magnetic field may act as a stabilizing factor on MHD boundary layer.  The stabilizing effect was justified by \cite{MR3882222} for the case with both viscosity and resistivity,  and by \cite{MR4102162}  for the case without resistivity.
 \end{remark}

 \section{Derivation of the boundary layer system}\label{secderivation}

This section is devoted to deriving the boundary layer system \eqref{mhdsys+++}.  We consider the MHD system in $\Omega$
\begin{equation}\label{mhdsys}
\left\{
\begin{aligned}
&\partial_t \bm  u^{\varepsilon} +(\bm  u^{\varepsilon} \cdot\nabla) \bm u^{\varepsilon}-(\bm H^{\varepsilon} \cdot\nabla) \bm H^{\varepsilon} +\nabla   p^{\varepsilon} =0,\\
&\partial_t \bm  H^{\varepsilon}  +(\bm u^{\varepsilon}\cdot\nabla)\bm H^{\varepsilon} -(\bm H^{\varepsilon}\cdot\nabla)\bm u^{\varepsilon}-\mu\varepsilon \Delta \bm H^{\varepsilon}=0,\\
&\nabla\cdot \bm u^{\varepsilon}=\nabla\cdot \bm H^{\varepsilon}=0,\\
&{\bm u^{\varepsilon}}|_{t=0}= {\bm u_0} ,\quad {\bm H^{\varepsilon}}|_{t=0}={\bm b_{0}},
\end{aligned}
\right.
\end{equation}
where $ \bm u^{\varepsilon}=(u^{\varepsilon},v^{\varepsilon}), \bm H^{\varepsilon}=(f^{\varepsilon},g^{\varepsilon}) $ denote velocity and  magnetic fields, respectively.  The above system is complemented with   the no-slip boundary condition on the normal component of velocity field and perfectly conducting boundary condition on the magnetic field,   that is,
\begin{equation}\label{cond2}
 v^{\varepsilon}|_{y=0}= (0),\quad (\partial_{y} f^{\varepsilon},  g^{\varepsilon}) |_{y=0}=(0,0).
\end{equation}
A boundary
layer  will appear in order to overcome  a mismatch on the boundary $ y=0 $ for the tangential magnetic fields between  \eqref{mhdsys}
 and the limiting equations by letting $\eps\rightarrow 0.$      To derive the governing equations for boundary layers  we consider the ansatz

\begin{equation}\label{expan}
\left\{
\begin{aligned}
	 &u ^ { \varepsilon } ( t , x , y ) = u ^ { 0 } ( t , x , y ) + u ^ { b } ( t , x , \tilde { y } ) + O ( \sqrt { \varepsilon }  ) ,\\ 
	 &v ^ { \varepsilon } ( t , x , y ) = v ^ { 0 } ( t , x , y ) + \sqrt { \varepsilon }   v ^ { b } ( t , x , \tilde { y } )   + O (  \varepsilon   ), \\ 
	 &f ^ { \varepsilon } ( t , x , y ) = f ^ { 0 } ( t , x , y ) + f ^ { b } ( t , x , \tilde { y } ) + O ( \sqrt { \varepsilon }  ), \\ 
	 &g ^ { \varepsilon } ( t , x , y ) = g ^ { 0 } ( t , x , y ) + \sqrt { \varepsilon }   g ^ { b } ( t , x , \tilde { y } )  + O (   \varepsilon  ), \\
	 &p ^ { \varepsilon } ( t , x , y ) = p ^ { 0 } ( t , x , y ) + p ^ { b } ( t , x , \tilde { y } ) + O ( \sqrt { \varepsilon }  ) ,
\end{aligned}
\right.
\end{equation}
where  we used the notation $ \tilde{ y } =y/ \sqrt{\varepsilon}$.   We suppose  $u^b, f^b, $ and $p^b$  in the expansion \eqref{expan}
polynomially trend  to zero as $\tilde y\rightarrow+\infty$, that is, as $\eps\rightarrow 0.$  Similarly for the expansion of the initial data.

\subsection*{Boundary conditions} 

Taking trace on $y=0$ for the second and the fourth expansions in \eqref{expan} and recalling the boundary condition \eqref{cond2},   we derive that 
\begin{equation}\label{bv0}
	v^0|_{y=0}=g^0|_{y=0}=0,
\end{equation} 
and using again the second and the fourth equations in \eqref{expan} and letting $\eps\rightarrow 0$, we get that
\begin{equation}\label{bv1}
	v^b|_{y=0}=g^b|_{y=0}=0.
\end{equation} 
Moreover observe 
\begin{eqnarray*}
	0=\partial_y f^\eps|_{y=0} =\partial_y f^0 |_{y=0}+\frac{1}{\sqrt \eps} \partial_{\tilde y}f^b|_{\tilde y=0}+o(1).
\end{eqnarray*}
This gives
\begin{equation}\label{bv2}
	\partial_{\tilde y}f^b|_{\tilde y=0}=0.
\end{equation}

\subsection*{The governing equations of the fluid behavior near and far from the boundary} We substitute the ansatz \eqref{expan} into \eqref{mhdsys} and  consider the order of $\eps$.
At the order $\eps^{-1/2}$   we get
\begin{equation*}
	\partial_{\tilde y} p^b\equiv 0.
\end{equation*}
This with the assumption that $p^b$ goes to $0$ as $\tilde y\rightarrow+\infty$ implies
\begin{equation}\label{np}
	p^b\equiv 0.
\end{equation} 
At the order $\eps^0$, letting 
 $\tilde y\rightarrow+\infty $ $ (\eps\rightarrow 0)$   and taking into account \eqref{bv0} and fact that   $u^b, f^b, $ and $p^b$   polynomially 
  trend  to zero  as $\eps\rightarrow 0$,  we see the limiting   system  is the ideal
incompressible MHD system:
\begin{equation}\label{imhdsys}
\left\{
\begin{aligned}
&\partial_t \bm  u^{0} +(\bm  u^{0} \cdot\nabla) \bm u^{0}-(\bm H^{0} \cdot\nabla) \bm H^{0} +\nabla   p^{0}=0,\\
&\partial_t \bm H^{0}  +(\bm u^{0}\cdot\nabla)\bm H^{0} -(\bm H^{0}\cdot\nabla)\bm u^{0}=0,\\
&\nabla\cdot \bm u^{0}=\nabla\cdot \bm H^{0}=0,\\
\end{aligned}
\right.
\end{equation}   
complemented with the boundary condition \eqref{bv0} and initial data $\bm u^0_{in}$ and $\bm H^0_{in}$, 
    where $\bm  u^{0}=(u^0, v^0), \bm H^{0}=( f^0, g^0)$. 
    
    Next we will derive the boundary layer equations.  Let  $\bm  u^{0}=(u^0, v^0), \bm H^{0}=( f^0, g^0)$ be the solution to the ideal MHD system \eqref{imhdsys}.   By Taylor expansion  we write $u^0(t,x,y)$ as  
    \begin{eqnarray*}
    	u^0(t,x,y)=u^0(t,x,0)+y\partial_y u^0(t,x,0)+\frac{y^2}{2}\partial_y^2 u^0(t,x,0)+\cdots=\overline{u^0}+\sqrt{\eps}\tilde {y} \overline{\partial_y u^0}+O(\eps),
    \end{eqnarray*}
   where here and below we use the notation $ \overline{h}$
to stand for the trace of a function $ h $ on the
boundary $ {y = 0} $.  Similarly,
    \begin{equation*}
    \begin{aligned}
    	& v^0(t,x,y) =\sqrt{\eps}\tilde {y} \overline{\partial_y v^0}+O(\eps),\quad f^0(t,x,y)= \overline{f^0}+\sqrt{\eps}\tilde {y} \overline{\partial_y f^0}+O(\eps),\\
    	&g^0(t,x,y)=  \sqrt{\eps}\tilde {y} \overline{\partial_y g^0}+O(\eps), \quad 	p^0(t,x,y) =\overline{p^0}+\sqrt{\eps}\tilde {y} \overline{\partial_y g^0}+O (\eps).\\
\end{aligned}
\end{equation*}
Now we compare  the order $\eps^0$  for the resulting equation by   substituting   the ansatz \eqref{expan}  as well as    the above Taylor expansion of $\bm  u^{0}=(u^0, v^0),  \bm H^{0}=( f^0, g^0)$  into \eqref{mhdsys}; this gives, by virtue of  \eqref{np}  and \eqref{imhdsys} ,  
\begin{equation}\label{limeq}
\left\{
\begin{aligned}
&\partial_t(\overline{u^0}+u^b)+(\overline{u^0}+u^b)\partial_x(\overline{u^0}+u^b)+(\tilde {y}\cdot \overline{\partial_y v^0}+v^b)\cdot\partial_{\tilde{y}}u^b\\
&\qquad-(\overline{f^0}+f^b)\partial_x(\overline{f^0}+f^b)-(\tilde {y}\cdot \overline{\partial_y g^0}+g^b)\cdot\partial_{\tilde{y}}f^b+\partial_x \overline{p^0} =0,
\\&
\partial_t(\overline{f^0}+f^b)+(\overline{u^0}+u^b)\partial_x(\overline{f^0}+f^b)+(\tilde {y}\cdot \overline{\partial_y v^0}+v^b)\cdot\partial_{\tilde{y}}f^b\\&\qquad-(\overline{f^0}+f^b)\partial_x(\overline{u^0}+u^b)-(\tilde {y}\cdot \overline{\partial_y g^0}+g^b)\cdot\partial_{\tilde{y}}u^b -\mu\partial_{\tilde{ y }}^{2}f^b=0,\\
& \partial_x (\overline{u^0}+u^b)+\partial_y (\tilde {y}\cdot \overline{\partial_y v^0}+v^b)=\partial_x (\overline{f^0}+f^b)+\partial_y (\tilde {y}\cdot \overline{\partial_y g^0}+g^b)=0.
\end{aligned}
\right.
\end{equation}
Denoting
\begin{equation*}
\begin{aligned}
&u ( t , x , \tilde y  ) =\overline{ u ^ { 0 }}  + u ^ { b } ( t , x , \tilde { y } ) ,\quad v ( t , x ,\tilde y ) =\tilde y \partial_{y} \overline{v ^ { 0 } } + v ^ { b } ( t , x , \tilde { y } ),  \\ 
&f  ( t , x ,\tilde y ) =\overline{ f ^ { 0 } } + f ^ { b } ( t , x , \tilde { y } ) ,\quad g  ( t , x ,\tilde y ) = \tilde y \partial_{y}\overline{g ^ { 0 } } + g ^ { b } ( t , x , \tilde { y } ),
\end{aligned}
\end{equation*}
and recalling $u ^ { b }, f^b$ polynomially  trend to $0$ as $\tilde y \rightarrow +\infty$,  
we combine \eqref{limeq}  and the boundary conditions \eqref{bv0}--\eqref{bv2} to conclude that all equations except the third one  in \eqref{mhdsys+++} are fulfilled by  $ u, v, f,$ and $g$. For simplicity of notation, we have replaced $ \tilde{y}  $ by $ y $. Note that the third equation in \eqref{mhdsys+++} can be derived from  the second one and the boundary condition $\partial_yf|_{y=0}=0$ by observing that
\begin{eqnarray*}
	g(t,x,y)=-\int_0^y \partial_x f (t,x,z)dz.
\end{eqnarray*}
Finally we remark   
 the Bernoulli's law \eqref{berlaw} follows by taking trace on $y=0$ for the ideal MHD system \eqref{imhdsys}.

\section{A priori  energy estimates}\label{sec}

The general strategy for constructing  solutions to   \eqref{mhd+} involves mainly two  ingredients. One is to construct appropriate  approximate solutions, which reserve a similar properties as \eqref{inco} for initial data  by applying  the standard  maximum principle for parabolic equations in the domain $\Omega$ (see \cite[Lemmas E.1  and E.2 ]{MR3385340} for instance).  Next   we need to deduce the uniform estimate   for these approximate solutions.    For sake of  simplicity we only present the following a priori estimate for regular solutions,  which is a key part to prove the main result Theorem \ref{pro1+}.  

\begin{theorem}\label{thma}
Let $ \ell >\frac 12 $ and $\delta>\ell+\frac 12$ be two given  numbers, and let $u,f\in L^\infty([0,T];  H_\ell^4)$ solve the MHD boundary layer system \eqref{mhd+} satisfying that a constant $c>0$ exists such that for any $(t,x,y)\in [0,T]\times \Omega$,    
	\begin{eqnarray*}
	   f(t, x,y) \geq  c  \comi y^{-\delta}   \ {\rm  and } \   \sum_{j\leq 2} \abs{\partial_y^j f(t, x,y)}\leq  c^{-1} \comi y^{-\delta-j}.	
	\end{eqnarray*}
Then there exists a constant $C>0$ such that
	\begin{eqnarray*}
		\mathcal E(t)+\int_0^t\mathcal D(s)ds \leq C \Big(\mathcal E(0)+\int_0^t \inner{\mathcal E(s)+\mathcal E(s)^2}ds\Big),
	\end{eqnarray*}
	where here and below 
	\begin{equation}\label{noes}
		\mathcal E(t):=\norm{u(t)}_{H_\ell^4}^2+\norm{f(t)}_{H_\ell^4}^2 ,\quad 
 \mathcal D(t) :=   \norm{  \partial_y f(t) } _ { H_\ell^4 } ^ { 2 }.
	\end{equation}
\end{theorem}
We will present the proof of Theorem \ref{thma}  in the next  two subsections,  one of which is  devoted to the estimates on tangential and another to the normal  derivatives.    To simplify the notation we will use the capital letter $C$ in the following argument to denote some generic constants that may vary from line to line,    and moreover use $C_\eps$  to denote some generic constants depending on a given number $0<\eps\ll 1.$ 

\subsection{Energy estimates:  tangential derivatives}  In this part, we will derive the estimate on tangential derivatives,  following  the  cancellation mechanism observed in the previous work of Liu-Xie-Yang \cite{MR3882222}.   
\begin{lemma}\label{pro2} Under the same assumption as in Theorem \ref{thma}  we have, for any $t\in[0,T],$
	\begin{multline*} 
	\sum_{   i\leq 4 } \inner{\norm{ \comi y^{\ell} \partial_x^i u(t)} _ { L ^ { 2 } } ^ { 2 } 
  +\norm{ \comi y^{\ell} \partial_x^i f(t)} _ { L ^ { 2 } } ^ { 2 }} \\
  + \sum_{  i\leq 4 } \int_0^t  \norm{ \comi y^{\ell} \partial_x^i\partial_y f (s)} _ { L ^ { 2 } } ^ { 2 }ds \leq   C \Big(\mathcal E(0)+\int_0^t \inner{\mathcal E(s)+\mathcal E(s)^2}ds\Big).
	\end{multline*}	
	Recall $\mathcal E$ is defined in \eqref{noes}.
\end{lemma} 
 
 \begin{proof}
Without loss of generality we may consider $i=4$ , 
apply  $\partial_x^4$  to the first second equations  and  $\partial_x^{3}$ to the third equation in \eqref{mhd+},  respectively, this gives
\begin{eqnarray}\label{sypv+}
\left
\{
\begin{aligned}
&\big(\partial_t+u\partial_x +v\partial_y\big)  \partial_x^4  u -\big(f\partial_x +g\partial_y \big)\partial_x^4 f= -(\partial_y u)\partial_x^4v+(\partial_y f)\partial_x^{4} g +F_{4},\\
&\big(\partial_t  +u\partial_x +v\partial_y  - \partial_y^2\big) \partial_x^4 f-\big(f\partial_x +g\partial_y \big)\partial_x^4 u
=-(\partial_y f)\partial_x^4 v+(\partial_y u)\partial_x^4 g+P_4,\\
&\big(\partial_t  +u\partial_x +v\partial_y  - \partial_y^2\big) \partial_x^{3} g
=f\partial_x^4 v-g\partial_x^4 u +Q_4,
\end{aligned}
\right.
\end{eqnarray}
where
\begin{equation*} 
F_{4}=\sum_{j=1}^4{4\choose j}\big[(\partial_x^jf)\partial_x^{5-j}f-(\partial_x^ju)\partial_x^{5-j}u\big]+\sum_{j=1}^{3}{4\choose j}\big[(\partial_x^jg)\partial_x^{4-j}\partial_yf-(\partial_x^jv)\partial_x^{4-j}\partial_yu \big],
\end{equation*}
\begin{equation*}
\begin{aligned}
P_4=&\sum_{j=1}^4{4\choose j}\big[(\partial_x^jf)\partial_x^{5-j}u-(\partial_x^ju)\partial_x^{5-j}f\big]+\sum_{j=1}^{3}{4\choose j}\big[(\partial_x^jg)\partial_x^{4-j}\partial_yu-(\partial_x^jv)\partial_x^{4-j}\partial_yf \big]\\
\end{aligned}
\end{equation*}
and
\begin{equation*}
\begin{aligned}
Q_4=&\sum_{j=1}^{3}{3\choose j}\big[(\partial_x^jf)\partial_x^{4-j}v-(\partial_x^jv)\partial_x^{3-j}\partial_y g\big]-\sum_{j=1}^{3}{3\choose j}\big[(\partial_x^ju)\partial_x^{4-j}g+(\partial_x^jg)\partial_x^{4-j}u \big] \\
=&\sum_{j=1}^{3}{3\choose j}\big[(\partial_x^jf)\partial_x^{4-j}v+(\partial_x^jv)\partial_x^{4-j}f\big]-\sum_{j=1}^{3}{3\choose j}\big[(\partial_x^ju)\partial_x^{4-j}g+(\partial_x^jg)\partial_x^{4-j}u \big].
\end{aligned}
\end{equation*}
In order to eliminate the terms  $ \partial_x^4 v$ and $ \partial_x^4 g $ where  the fifth order tangential derivatives are involved,  we introduce, observing $f>0$ by assumption,   
\begin{eqnarray}\label{plam}
\psi  \stackrel{\rm def}{ =}  \partial_x^4 f+\frac{\partial_y f}{f}\partial_x ^{3}g=-f\partial_{ y }\big(\frac{\partial_{ x }^{3}g}{f} \big), \quad \varphi  \stackrel{\rm def}{ =}  \partial_x^4 u+\frac{\partial_y u}{f}\partial_x ^{3}g.
\end{eqnarray}
Multiplying  the third equation in \eqref{sypv+} by $(\partial_y f)/f $  and then taking summation with the second one in \eqref{sypv+},  we obtain the equation solved by $\psi$, that is,
\begin{equation}\label{eqpsi}
	\big(\partial_t  +u\partial_x +v\partial_y  - \partial_y^2\big) \psi-\big(f\partial_x +g\partial_y \big)\varphi=L_{4},
\end{equation}
where 
\begin{multline*}
	L_{4}
= P_{4}+\frac{\partial_{ y }f}{f}Q_4-\frac{g(\partial_{ y }f)}{f}\partial_x^4 u +\Big[\frac{g(\partial_{ y }u)}{f}+2\partial_{ y}\big( (\partial_y f)/ f\big)\Big]\partial_x^{4}f\\ 
 +\Big[\frac{2(\partial_y f)\partial_{ y}^{2}f}{f^{2}}+\frac{(\partial_{ x }u)\partial_y  f}{f}-\frac{2(\partial_y f)^{3}}{f^{3}}-\frac{(\partial_{ y }u)\partial_x f}{f}   \Big]\partial_x^{3} g
\end{multline*}
with $P_4$ and $Q_4$ given in \eqref{sypv+}.   Similarly we multiply   the third equation in \eqref{sypv+}  by $(\partial_y u)/f $  and then add the resulting equation by  the first one in \eqref{sypv+},  to obtain
\begin{eqnarray}\label{sypv++}
\big(\partial_t  +u\partial_x +v\partial_y  \big) \varphi-\big(f\partial_x +g\partial_y \big)\psi=-\frac{\partial_{ y }u}{f}\partial_y\psi  +M_{4},
\end{eqnarray}
 where 
\begin{equation*}  
\begin{aligned}
M_{4}=&F_{4}+\frac{\partial_{ y }u}{f}Q_4-\frac{g(\partial_{ y }u)}{f} \partial_x^{4}u+\Big[\frac{g(\partial_{ y }f)}{f} +\frac{(\partial_yu)\partial_y f}{f})\Big]\partial_x^{4}f\\
 &\quad+\Big[\frac{(\partial_{ y}f)\partial_x f}{f} +\frac{g(\partial_y f)^{2}}{f^{2}} -\frac{(\partial_y f)^{2}\partial_{ y }u}{f^{3}}-\frac{(\partial_{ x }u)\partial_y  u}{f}-\frac{g(\partial_y u)^{2}}{f^{2}}  \Big]\partial_x^{3} g
\end{aligned}
\end{equation*}
with $F_4$ and $Q_4$ given in \eqref{sypv+}.   Note that \eqref{eqpsi} is complemented with the boundary condition that 
\begin{eqnarray*}\label{inplam}
\partial_y\psi |_{y=0} =0.
\end{eqnarray*}
Thus we perform  the weighted energy estimate  for \eqref{eqpsi}  and \eqref{sypv++}  and use the fact that
\begin{eqnarray*}
	 \big( \big(f\partial_x +g\partial_y \big)\comi y^{\ell}\varphi  ,\   \comi y^{\ell} \psi \big)_{L^2}+ \big( \big(f\partial_x +g\partial_y \big)\comi y^{\ell}\psi  ,\   \comi y^{\ell} \varphi \big)_{L^2}=0
\end{eqnarray*}
to get
 \begin{equation}\label{eqentan}
\begin{aligned} 
& \frac { 1 } { 2 } \frac { d } { d t } \inner{\norm{  \comi y^{\ell} \psi  } _ { L ^ { 2 } } ^ { 2 }+\norm{\comi y^{\ell} \varphi } _ { L ^ { 2 } } ^ { 2 }}+\norm{   \partial _ { y } (\comi y^{\ell} \psi)  } _ { L ^ { 2 } } ^ { 2 } \\
& = \big( \comi y^{\ell} L_4 ,\   \comi y^{\ell} \psi \big)_{L^2} +\big( \comi y^{\ell} M_4 ,\   \comi y^{\ell} \varphi \big)_{L^2} -\big( \comi y^{\ell} \big((\partial_yu)/f)\partial_y\psi,\   \comi y^{\ell} \varphi \big)_{L^2}\\
&\quad+ \big( [v \partial_y,\ \comi y^{\ell}]\psi ,\   \comi y^{\ell} \psi \big)_{L^2} - \big(  [\partial_y^2, \ \comi y^{\ell}] \psi ,\   \comi y^{\ell} \psi \big)_{L^2}- \big( [g \partial_y, \ \comi y^{\ell}]\varphi ,\   \comi y^{\ell} \psi \big)_{L^2}\\
&\quad+  \big( [v \partial_y,\ \comi y^{\ell}]\varphi ,\   \comi y^{\ell} \varphi \big)_{L^2}  - \big( [g \partial_y, \ \comi y^{\ell}]\psi ,\   \comi y^{\ell} \varphi \big)_{L^2},
\end{aligned}
\end{equation}
where here and below we use $[\mathcal T_1,\  \mathcal T_2]$ to denote the commutator between two operators $\mathcal T_1, \mathcal T_2$,  that is 
\begin{equation}\label{comta}
	[\mathcal T_1,\ \mathcal T_2]=\mathcal T_1 \mathcal T_2-\mathcal T_2 \mathcal T_1.
\end{equation}
Observe  the derivatives  are at most up to the fourth order  for the terms on the right of \eqref{eqentan}. Then by direct compute we have
\begin{eqnarray*}
\begin{aligned}
	& \big( \comi y^{\ell} L_4 ,\   \comi y^{\ell} \psi \big)_{L^2} +\big( \comi y^{\ell} M_4 ,\   \comi y^{\ell} \varphi \big)_{L^2} -\big( \comi y^{\ell} \big((\partial_yu)/f)\partial_y\psi,\   \comi y^{\ell} \varphi \big)_{L^2}\\
 &\quad+ \big( [v \partial_y,\ \comi y^{\ell}]\psi ,\   \comi y^{\ell} \psi \big)_{L^2} - \big(  [\partial_y^2, \ \comi y^{\ell}] \psi ,\   \comi y^{\ell} \psi \big)_{L^2}- \big( [g \partial_y, \ \comi y^{\ell}]\varphi ,\   \comi y^{\ell} \psi \big)_{L^2}\\
 &\qquad+  \big( [v \partial_y,\ \comi y^{\ell}]\varphi ,\   \comi y^{\ell} \varphi \big)_{L^2}  - \big( [g \partial_y, \ \comi y^{\ell}]\psi ,\   \comi y^{\ell} \varphi \big)_{L^2}\leq \frac{1}{2}\norm{   \partial _ { y } (\comi y^{\ell} \psi)  } _ { L ^ { 2 } } ^ { 2 }  +C\inner{\mathcal E+\mathcal E^{2}}.
 \end{aligned}
\end{eqnarray*}
Substituting the above estimate into \eqref{eqentan}  and then integrating over $[0,t]$ for any $0<t<T$ gives
\begin{equation}\label{fes}
	 \norm{  \comi y^{\ell} \psi (t) } _ { L ^ { 2 } } ^ { 2 }+\norm{\comi y^{\ell} \varphi (t)} _ { L ^ { 2 } } ^ { 2 } +\int_0^t \norm{   \partial _ { y } (\comi y^{\ell} \psi)  } _ { L ^ { 2 } } ^ { 2 } ds\leq C\Big(\mathcal E(0)+  \int_0^t \inner{\mathcal E(s)+\mathcal E(s)^2}ds\Big). 
\end{equation}
Next we will derive the estimates for $f, u$ from the ones of $\psi,\varphi$.  In fact in view of the representation of $\psi$ given in \eqref{plam},  we use Hardy-type inequality (cf.  \cite[Lemma B.1]{MR3385340} for instance) to conclude 
\begin{eqnarray*}
	\norm{\comi y^{\ell-1} \partial_x^3 g}_{L^2} \leq   C \norm{\comi y^{\ell}  \psi}_{L^2}.
\end{eqnarray*}
As a result, using  the representation of $\psi$ and $\varphi$ given in \eqref{plam} and the fact that $|(\partial_yf)/f |\lesssim \comi y^{-1}$ gives
\begin{eqnarray*}
		\norm{\comi y^{\ell} \partial_x^4 f}_{L^2} \leq C\norm{\comi y^{\ell-1} \partial_x^3 g}_{L^2} +\norm{\comi y^{\ell}   \psi}_{L^2} \leq  C  \norm{\comi y^{\ell}   \psi}_{L^2}
\end{eqnarray*}
and
\begin{eqnarray*}
		\norm{\comi y^{\ell} \partial_x^4 u}_{L^2} \leq C\norm{\comi y^{\ell-1} \partial_x^3 g}_{L^2} +\norm{\comi y^{\ell}   \varphi}_{L^2} \leq  C  \norm{\comi y^{\ell}   \psi}_{L^2}+\norm{\comi y^{\ell}   \varphi}_{L^2}. 
\end{eqnarray*}
Moreover,  using again \eqref{plam},
\begin{eqnarray*}
	\norm{\comi y^{\ell} \partial_y \partial_x^4 f}_{L^2} \leq C\norm{\comi y^{\ell} \partial_y\big[((\partial_yf)/f) \partial_x^3 g]}_{L^2} +\norm{\comi y^{\ell} \partial_y  \psi}_{L^2} \leq  C  \norm{ \partial_y (\comi y^{\ell}   \psi)}_{L^2}+C\mathcal E^{1/2}.
\end{eqnarray*}
Combining these inequality with \eqref{fes} we conclude 
\begin{eqnarray*}
	 \norm{  \comi y^{\ell} \partial_x^4 u (t) } _ { L ^ { 2 } } ^ { 2 }+\norm{  \comi y^{\ell} \partial_x^4 f (t) } _ { L ^ { 2 } } ^ { 2 } +\int_0^t \norm{\comi y^{\ell} \partial_x^4 \partial_y f} _ { L ^ { 2 } } ^ { 2 } ds \leq C\Big(\mathcal E(0)+  \int_0^t \inner{\mathcal E(s)+\mathcal E(s)^2}ds\Big). 
\end{eqnarray*}
Note the above estimate still holds true if we replace $ \partial_x^4$ by $\partial_x^i$ with $i\leq 4$.   This gives the desired estimate in  Lemma  \ref{pro2},  completing the proof.
\end{proof}

  \subsection{Estimate for normal derivatives} 
In this part, we perform the estimate for normal derivatives. Compared with   \cite{MR3882222}  a new difficulty arises when dealing the boundary integrals  because of the absence of  the hydrodynamic viscosity.     

\begin{lemma}\label{pro1}
	Under the same assumption as in Theorem \ref{thma} we have, for any $t\in[0,T]$,
	 \begin{multline*}
	   \sum_{ \stackrel{ i+j\leq 4 }{  j\geq 1 }} \inner{\norm{ \comi y^{\ell+j} \partial_x^i\partial_y^{j}u(t)} _ { L ^ { 2 } } ^ { 2 } 
  +\norm{ \comi y^{\ell+j} \partial_x^i\partial_y^{j}f(t)} _ { L ^ { 2 } } ^ { 2 }}\\
  + \int_0^t \sum_{ \stackrel{ i+j\leq 4 }{  j\geq 1 }}  \norm{ \comi y^{\ell+j} \partial_x^i\partial_y^{j+1} f (s)} _ { L ^ { 2 } } ^ { 2 }ds\leq  \eps \int_0^t \mathcal D(s)ds+C_\eps \Big(\mathcal E(0)+\int_0^t \inner{\mathcal E(s)+\mathcal E^2(s)}ds\Big).  	
	 \end{multline*}
	 Recall $\mathcal E, \mathcal D$ are defined by \eqref{noes}.
\end{lemma} 

\begin{proof}  {\it Step 1).}    We first consider the case of $i=0$ and   $j=4$.  In this step we will prove that, for any $\eps>0$,
\begin{multline}\label{fet}
	  \inner{\norm{\comi y^{\ell+4}  \partial_y^4 u(t) } _ { L^ { 2 } } ^ { 2 } +\norm{\comi y^{\ell+4}  \partial_y^4 f(t) } _ { L^ { 2 } } ^ { 2 } }+\int_0^t \norm{ \comi y^{\ell+4}  \partial_y^5 f(s)  } _ { L^ { 2 } } ^ { 2 } ds \\
\leq  \int_0^t\Big(\int_{\mathbb{T}}( \partial_y^{4} f) (f \partial_y^{3} \partial_x u) |_{y=0}\,dx \Big) ds+\eps \int_0^t\mathcal D(s)ds+C_\eps\Big(\mathcal E(0) +\int_0^t \inner{\mathcal E(s)+\mathcal E^2(s)} ds\Big).   
\end{multline}
	Applying $\comi y^{\ell+4}\partial_y^4$ to \eqref{mhd+} yields that
	\begin{multline*}
		\big(\partial_t+u\partial_x +v\partial_y\big) \comi y^{\ell+4} \partial_y^4  u -\big(f\partial_x +g\partial_y \big) \comi y^{\ell+4} \partial_y^4 f\\
		= [u\partial_x +v\partial_y,\  \comi y^{\ell+4}\partial_y^4]u-[f\partial_x +g\partial_y, \  \comi y^{\ell+4}\partial_y^4]f
	\end{multline*}
	and
	\begin{multline*}
		\big(\partial_t  +u\partial_x +v\partial_y  - \partial_y^2\big) \comi y^{\ell+4}\partial_y^4 f-\big(f\partial_x +g\partial_y \big)\comi y^{\ell+4}\partial_y^4 u\\
	=[u\partial_x +v\partial_y,\ \comi y^{\ell+4}\partial_y^4]f-[f\partial_x +g\partial_y,\ \comi y^{\ell+4}\partial_y^4]u-[\partial_y^2,\ \comi y^{\ell+4}\partial_y^4]f.
	\end{multline*}
	Recall $[\cdot, \cdot]$ is given by \eqref{comta},  standing for the commutator between two operators. 
 Taking inner product with $ \comi y^{\ell+4} \partial_y^4  u$ to the first equation above, and with $ \comi y^{\ell+4} \partial_y^4  f$ to the second one,  and then taking summation and observing 
 \begin{eqnarray*}
 \Big(\big(f\partial_x +g\partial_y \big) \comi y^{\ell+4} \partial_y^4 f,   \  \comi y^{\ell+4} \partial_y^4 u\Big)_{L^2} 
 	 +\Big(\big(f\partial_x +g\partial_y \big) \comi y^{\ell+4} \partial_y^4 u,   \   \comi y^{\ell+4} \partial_y^4 f\Big)_{L^2}=0,
 \end{eqnarray*}
  we obtain 
 \begin{multline} \label{enes}
   \frac { 1 } { 2 } \frac { d } { d t } \inner{\norm{\comi y^{\ell+4}  \partial_y^4 u } _ { L^ { 2 } } ^ { 2 } +\norm{\comi y^{\ell+4}  \partial_y^4 f } _ { L^ { 2 } } ^ { 2 } }+\norm{\partial_y\big(\comi y^{\ell+4}  \partial_y^4 f\big) } _ { L^ { 2 } } ^ { 2 }  \\
=    - \int_{\mathbb{T}}(\partial_y^4 f )\partial_y^{5} f  |_{y=0}  dx + R_4
 \end{multline}
 with 
 \begin{multline*}
 	 R_4= \Big( [u\partial_x +v\partial_y,\  \comi y^{\ell+4}\partial_y^4]u-[f\partial_x +g\partial_y, \  \comi y^{\ell+4}\partial_y^4]f,   \  \comi y^{\ell+4} \partial_y^4 u\Big)_{L^2}\\
 	+ \Big([u\partial_x +v\partial_y,\ \comi y^{\ell+4}\partial_y^4]f-[f\partial_x +g\partial_y,\ \comi y^{\ell+4}\partial_y^4]u-[\partial_y^2,\ \comi y^{\ell+4}\partial_y^4]f,   \  \comi y^{\ell+4} \partial_y^4 f\Big)_{L^2}.
 \end{multline*}
 Direct computation shows
 \begin{equation}\label{err}
 	  R_4\leq \frac{1}{2}  \norm{\partial_y\big(\comi y^{\ell+4}  \partial_y^4f\big) } _ { L^ { 2 } } ^ { 2 } +C \inner{\mathcal E +\mathcal E ^{2}}.
 \end{equation}
 It remains to   deal with the boundary integeral on the right of  \eqref{enes}.   We first apply  $\partial_y$ to the second equation in \eqref{mhd+} and then take trace on $y=0$; this together with  the boundary condition in \eqref{mhd+}  gives
 \begin{equation}\label{b3}
 	 \partial_y^{3} f|_{y=0}=2(\partial_y u) \partial_x f|_{y=0}-f\partial_x\partial_yu|_{y=0}.
 \end{equation}
By virtue of the above representation  of  $\partial_y^3f|_{y=0}$,  we compute directly that  
 \begin{eqnarray}\label{bound}
\begin{aligned}
\partial_y^{5} f|_{y=0}& =\partial_y^{3}\big(\partial_t f+u\partial_x f +v\partial_yf-f\partial_x u -g\partial_yu \big) \big|_{y=0} \\
&= \Big\{u\partial_x\partial_y^3 f-f\partial_x\partial_y^3 u-4(\partial_x u)
\partial_y^3 f-7(\partial_x\partial_y  u)\partial_y^2 f+4(\partial_x f)\partial_y^3  u\\
&\qquad +8(\partial_y  u)\partial_x\partial_y^2  f-u(\partial_xf)\partial_x\partial_y u+uf\partial_x^2\partial_y u-2u(\partial_y u)\partial_x^2f
+2f(\partial_y u)\partial_x^2u
\Big\} \Big|_{y=0}.
\end{aligned}
\end{eqnarray}
As a result we combine \eqref{bound} with Sobolev's inequality to conclude
\begin{eqnarray*}
	   - \int_{\mathbb{T}}(\partial_y^4 f )\partial_y^{5} f  |_{y=0}  dx\leq    \int_{\mathbb{T}} f(\partial_y^4 f )(\partial_x\partial_y^{3} u ) |_{y=0}  dx+\eps \mathcal D+C_\eps \inner{\mathcal E+\mathcal E^2}.
\end{eqnarray*}
Substituting the above inequality and \eqref{err} into  \eqref{enes} and then integrating over $[0, t]$ for any $t\in[0, T]$ we obtain the desired estimate \eqref{fet}.

\bigskip
{\it Step 2).} In this step we will treat the first term on the right of \eqref{fet}  and prove that
\begin{equation}
	\label{bouest}
	 \int_0^t\Big( \int_{\mathbb{T}}( \partial_y^{4} f) (f \partial_y^{3} \partial_x u) |_{y=0}\,dx\Big) ds\leq \eps \int_0^t\mathcal D(s)ds+C_\eps   \int_0^t  \mathcal E^2(s) ds 
\end{equation}
holds true for any $\eps>0$. 
  To do so we recall some facts on the Fourier multiplier.  
Let $k\in\mathbb Z$ be the partial Fourier dual variable of $x\in\mathbb T$ and let $\Lambda_x^\sigma,\sigma\in\mathbb R,$ be the Fourier multiplier with symbol $\inner{1+k^2}^{\sigma/2}$, that is,
\begin{eqnarray*}
	\mathcal F_x (\Lambda_x^\sigma f)(k)=(1+k^2)^{\sigma/2}  \mathcal F_x (f)(k),
\end{eqnarray*}
where $\mathcal F_x$ stands for the Fourier transform with respect to $x$ variable: 
\begin{eqnarray*}
	(\mathcal F_x f)(k):=\int_{\mathbb T}f(x)e^{-i kx}dx,\quad
	k\in\mathbb Z. \end{eqnarray*}
Similarly  we define $|D_x|^\sigma, \sigma>0,$ by setting 
\begin{eqnarray*}
	\mathcal F_x (\abs{D_x}^\sigma f)(k)= \abs k ^{\sigma}  \mathcal F_x (f)(k). 
\end{eqnarray*}
Given a  $C^1$ function $\rho$ of $x\in\mathbb T$ with bounded derivatives,  we have, for $0<\sigma<1$,
\begin{equation}\label{coe}
	\forall\ w\in L_x^2 , \quad  \|[\abs{D_x}^\sigma , \ \rho] w  \|_{L^2_x}\leq C_\sigma\inner{\norm{\rho}_{L^\infty}+\norm{\partial_x\rho}_{L^\infty}}\norm{w}_{L^2_x} 	
	\end{equation} 
with  $C_\sigma$ a constant depending only on $\sigma$, recalling the commutator $[\abs{D_x}^\sigma , \ \rho]$ is defined by \eqref{comta}.   Note the   counterpart for $x\in\mathbb R$ of \eqref{coe}  is clear,  see, e.g., \cite[Pages 702--704]{MR2763329},   the estimate \eqref{coe} can be proven in a similar inspirit  and we omit it for brevity and  and refer to \cite{MR2308505}  and references therein for  the  comprehensive  argument on the extension of the classical  pseudo-differential calculus  in $\mathbb R$ to the torus case $x\in\mathbb T$.    

With the Fourier multipliers introduced above we  use \eqref{coe} to compute 
\begin{equation}\label{ebound4}
	\begin{aligned}  
	&\Big|\int_{\mathbb{T}}( \partial_y^{4} f) (f \partial_y^{3} \partial_x u) |_{y=0}\,dx \Big|  \leq  \norm{ \Lambda_x^{1/2}  \partial_y^{3}u|_{y=0} }_{L_x^2} \norm{  \Lambda_x^{1/2} ( f\partial_y^{4}f)|_{y=0}}_{L_x^2} \\
	&\leq  C \norm{ \Lambda_x^{1/2}  \partial_y^{3}u|_{y=0} }_{L_x^2} \inner{\norm{  \abs{D_x}^{1/2} ( f\partial_y^{4}f)|_{y=0}}_{L_x^2}+\norm{  ( f\partial_y^{4}f)|_{y=0}}_{L_x^2}}\\
	&\leq  C \norm{ \Lambda_x^{1/2}  \partial_y^{3}u|_{y=0} }_{L_x^2} \inner{\norm{  \partial_xf}_{L^\infty} +\norm{  f}_{L^\infty} }   \norm{  \Lambda_x^{1/2}   \partial_y^{4}f|_{y=0} }_{L_x^2} \\
	& \leq  \eps\norm{  \Lambda_x^{1/2}   \partial_y^{4}f|_{y=0} }_{L_x^2}^2 +  C_\eps  \mathcal E \norm{ \Lambda_x^{1/2}  \partial_y^{3}u|_{y=0}}_{L_x^2} ^2. 
	\end{aligned}
	\end{equation}
On the other hand, using the fact that
\begin{eqnarray*}
	 \inner{\Lambda_x^{1/2}  \partial_y^{3}u(x,0)}^2=- 2 \int_0^{+\infty}   \big(\Lambda_x^{1/2}  \partial_y^{3}u(x,  y) \big) \Lambda_x^{1/2}  \partial_y^{4}u(x,  y) d  y,  
\end{eqnarray*}
 we compute
 \begin{eqnarray*}
 \begin{aligned}
 	\norm{ \Lambda_x^{1/2}  \partial_y^{3}u(x,0)}_{L_x^2} ^2&=- 2 \int_{\mathbb R}\Big(\int_0^{+\infty}   \big(\Lambda_x^{1/2}  \partial_y^{3}u(x,  y) \big) \Lambda_x^{1/2}  \partial_y^{4}u(x,  y) d  y\Big) dx\\
 	&=- 2\int_0^{+\infty}\Big(   \int_{\mathbb R} \big(\Lambda_x^{1/2}  \partial_y^{3}u(x,  y) \big) \Lambda_x^{1/2}  \partial_y^{4}u(x,  y) d  x\Big) dy\\
 	&\leq 2 \int_0^{+\infty}  \norm{ \Lambda_x  \partial_y^{3}u(\cdot,  y)}_{L_x^2}\norm{ \partial_y^{4}u(\cdot,  y)}_{L_x^2}   dy \leq   \mathcal E.   
 	\end{aligned}
 \end{eqnarray*}	
	A similar argument gives 
\begin{eqnarray*}
 \begin{aligned}
 	\norm{ \Lambda_x^{1/2}  \partial_y^{4}f(x,0)}_{L_x^2} ^2& \leq 2 \int_0^{+\infty}  \norm{ \Lambda_x  \partial_y^{4}f(\cdot,  y)}_{L_x^2}\norm{ \partial_y^{5}f(\cdot,  y)}_{L_x^2}   dy \leq   \mathcal D.   
 	\end{aligned}
 \end{eqnarray*}	
Substituting the two inequalities above into \eqref{ebound4} yields 
\begin{eqnarray*}
	\Big|\int_{\mathbb{T}}( \partial_y^{4} f) (f \partial_y^{3} \partial_x u) |_{y=0}\,dx \Big|  \leq  \norm{ \Lambda_x^{1/2}  \partial_y^{3}u }_{L_x^2L_y^\infty} \norm{  \Lambda_x^{1/2} ( f\partial_y^{4}f)}_{L_x^2L_y^\infty}\leq \eps \mathcal D+C_\eps \mathcal E^2.
\end{eqnarray*}
This gives the desired estimate \eqref{bouest}. 

\bigskip
{\it Step 3).}   We combine \eqref{fet} and \eqref{bouest} to obtain
\begin{multline*}\label{fet++}
	  \inner{\norm{\comi y^{\ell+4}  \partial_y^4 u(t) } _ { L^ { 2 } } ^ { 2 } +\norm{\comi y^{\ell+4}  \partial_y^4 f(t) } _ { L^ { 2 } } ^ { 2 } }+\int_0^t \norm{ \comi y^{\ell+4}  \partial_y^5 f(s)  } _ { L^ { 2 } } ^ { 2 } ds \\
\leq     \eps \int_0^t\mathcal D(s)ds+C_\eps\Big(\mathcal E(0) +\int_0^t \inner{\mathcal E(s)+\mathcal E^2(s)} ds\Big).   
\end{multline*}
Observe the  above inequality still holds true if we replace $\comi y^{\ell+4}  \partial_y^4$ by $\comi y^{\ell+j} \partial_x^i \partial_y^j$ with $i+j\leq 4$  and  use the boundary conditions  \eqref{b3} and  $\partial_y f|_{y=0}=0$.  Since the argument is  straightforward we omit it for brevity.  The proof of Lemma \ref{pro1} is completed.
\end{proof}

\subsection{Completing the proof of the energy estimate}  Combining the estimates in Lemmas \ref{pro2}--\ref{pro1} and letting $\eps$ be small enough we obtain the desired energy estimate, completing the proof of Theorem \ref{thma}.

\bigskip
\noindent {\bf Acknowledgements.} The work was supported by NSF of China(Nos. 11961160716, 11871054, 11771342) and   the Natural Science Foundation of Hubei Province (No. 2019CFA007).

%\bibliographystyle{abbrv}

%\bibliography{Bib-fluid}

\end{document}